\documentclass[11pt]{amsart}

\usepackage{amscd}
\usepackage{amsmath, amssymb}
\usepackage{amsfonts}
\newcommand{\de}{\partial}

\newcommand{\ov}[1]{\overline{#1}}

\newcommand{\ti}[1]{\tilde{#1}}
\newcommand{\vp}{\varphi}

\newcommand{\ve}{\varepsilon}

\renewcommand{\leq}{\leqslant}
\renewcommand{\geq}{\geqslant}

\newcommand{\be}{\begin{equation}}
\newcommand{\ee}{\end{equation}}

\begin{document}
\newtheorem{claim}{Claim}
\newtheorem{theorem}{Theorem}[section]
\newtheorem{lemma}[theorem]{Lemma}
\newtheorem{corollary}[theorem]{Corollary}
\newtheorem{proposition}[theorem]{Proposition}
\newtheorem{question}{question}[section]
\newtheorem{definition}[theorem]{Definition}
\newtheorem{remark}[theorem]{Remark}

\numberwithin{equation}{section}

\title[Curvature estimates]
{Curvature estimates for a class of Hessian type equations}

\author[J. Chu]{Jianchun Chu}
\address{Department of Mathematics, Northwestern University, 2033 Sheridan Road, Evanston, IL 60208}
\email{jianchun@math.northwestern.edu}

\author[H. Jiao]{Heming Jiao}
\address{School of Mathematics and Institute for Advanced Study in Mathematics, Harbin Institute of Technology,
         Harbin, Heilongjiang, 150001, China}
\email{jiao@hit.edu.cn}

\begin{abstract}
In this paper, we establish the curvature estimates for a class of Hessian type equations. Some applications are also discussed.
\end{abstract}

\maketitle

\section{Introduction}
Suppose that $M$ is a hypersurface in $(n+1)$-dimensional Euclidean space $\mathbb{R}^{n+1}$. Let $\kappa (X)$, $H(X)$ and $\nu (X)$
be the principal curvatures,  mean curvature and unit outer normal at $X \in M$ respectively. Define the $(0, 2)$-tensor field $\eta$ on $M$ by
\[
\eta_{ij} = H g_{ij}-h_{ij},
\]
where $g_{ij}$ and $h_{ij}$ are the first and second fundamental forms of $M$ respectively.
In fact, $\eta$ is the first Newton transformation of $h$ with respect to $g$. Using $\lambda(\eta)$ to denote the eigenvalues of $\eta$ (with respect to $g$), we see that
\[
\lambda_{i} = H-\kappa_{i} = \sum_{j\neq i}\kappa_{j}, \quad \text{for $i = 1,2 \cdots, n$}.
\]

In this paper, we consider the $k$-Hessian equation of $\lambda(\eta)$, i.e.,
\begin{equation}
\label{cj-1}
\sigma_k (\lambda (\eta)) = f (X, \nu (X)), \quad \text{for $X \in M$},
\end{equation}
where $\sigma_{k}$ is the $k$-th elementary symmetric function
\[
\sigma_{k} (\lambda) = \sum_ {i_{1} < \cdots < i_{k}}
\lambda_{i_{1}} \cdots \lambda_{i_{k}}, \quad \text{for $k=1,2,\cdots,n$}.
\]
To study equation (\ref{cj-1}), we introduce the following elliptic condition.
\begin{definition}
A $C^2$ regular hypersurface $M \subset \mathbb{R}^{n+1}$ is called $(\eta, k)$-convex if $\lambda(\eta) \in \Gamma_k$ for all $X \in M$,
where $\Gamma_k$ is the G{\aa}rding cone
\[
\Gamma_{k} = \{\lambda \in \mathbb{R}^{n}~|~ \sigma_{j} (\lambda) > 0, \  j = 1,2,\cdots, k\}.
\]
\end{definition}

Our main result is the following curvature estimate for equation \eqref{cj-1}.
\begin{theorem}
\label{Curvature estimate}
Let $M$ be a closed star-shaped $(\eta, k)$-convex hypersurface satisfying the curvature equation \eqref{cj-1}
for some positive function $f\in C^{2}(\Gamma)$, where $\Gamma$ is an open neighborhood of the unit normal bundle
of $M$ in $\mathbb{R}^{n+1}\times\mathbb{S}^{n}$. Then there exists a constant $C$ depending only
$n$, $k$, $\|M\|_{C^{1}}$, $\inf f$ and $\|f\|_{C^{2}}$ such that
\begin{equation}
\label{cj-3}
\max_{X\in M,~i=1,\cdots,n}|\kappa_{i}(X)| \leq C.
\end{equation}
\end{theorem}

If we replace $\lambda(\eta)$ by $\kappa(X)$ in \eqref{cj-1}, equation \eqref{cj-1} becomes the classical prescribed curvature equation
\begin{equation}
\label{cj-4}
\sigma_k (\kappa (X)) = f (X, \nu (X)), \quad \text{for $X \in M$}.
\end{equation}

When $k = 1,2$ and $n$, \eqref{cj-4} are prescribed mean curvature, scalar curvature and Gauss curvature equation respectively.
Establishing the global $C^2$ estimate for \eqref{cj-4} is a longstanding problem. When $k = 1$, it is quasi-linear and
the $C^2$ estimate follows from classical theory of quasi-linear PDEs. If $k=n$, \eqref{cj-4} is Monge-Amp\`{e}re type
and the $C^2$ estimates for general $f (X, \nu)$ are established by Caffarelli-Nirenberg-Spruck \cite{CNSI}. When $k = 2$, the $C^2$ estimate for \eqref{cj-4} was obtained by Guan-Ren-Wang \cite{GRW15} and their proof was simplified by Spruck-Xiao \cite{SX17}. In \cite{RW19,RW20}, Ren-Wang proved the $C^2$ estimate when $k=n-1$ and $n-2$.

When $2<k<n$, Caffarelli-Nirenberg-Spruck \cite{CNSIV} proved the $C^{2}$ estimate if $f$ is independent of $\nu$. Guan-Guan \cite{GG02} obtained the $C^{2}$ estimate if $f$ depends only on $\nu$. Ivochkina \cite{Ivochkina90,Ivochkina91} considered the corresponding Dirichlet problem of (\ref{cj-4}) on Euclidean domain and established the $C^{2}$ estimate under some extra assumptions on the dependence of $f$ on $\nu$. For equations of the prescribing curvature measure problem, when $f(X,\nu)=\langle X,\nu\rangle\ti{f}(X)$, the $C^{2}$ estimate was proved by Guan-Lin-Ma \cite{GLM09} and Guan-Li-Li \cite{GLL12}. For general $f(X,\nu)$, Guan-Ren-Wang \cite{GRW15} established such estimates for $(k+1)$-convex hypersurfaces (i.e., $\kappa(X) \in \Gamma_{k+1}$ for all $X\in M$).

When $k=n=2$, \eqref{cj-1} is the same as \eqref{cj-4}, which is the prescribed Gauss curvature equation. Thus \eqref{cj-1} can be regarded as a generalization of the classical prescribed curvature equation. When $k=n$, \eqref{cj-1} becomes the following equation for $(\eta,n)$-convex hypersurface:
\begin{equation}
\label{cj-5}
\det (\eta (X)) = f (X, \nu (X)), \quad \text{for $X \in M$}.
\end{equation}
The $(\eta,n)$-convex hypersurface has been studied intensively by Sha \cite{Sha86,Sha87}, Wu \cite{Wu87}  and Harvey-Lawson \cite{HL13}. We note that $(\eta, n)$-convexity was called $(n-1)$-convexity in \cite{Sha86,Sha87,HL13} (Here $(n-1)$-convexity is different from the above) or $(n-1)$-positivity in \cite{Wu87}. On the other hand, the left hand side of \eqref{cj-5} is a combination of Weingarten curvatures, which is a natural curvature function of $(\eta,n)$-convex hypersurfaces. So it is interesting to consider the curvature equation \eqref{cj-5} and its generalization \eqref{cj-1}.

In the complex setting, the corresponding Hessian type equation of (\ref{cj-1}) has been studied extensively. In particular, when $k=n$, it is called $(n-1)$ Monge-Amp\`ere equation, which is related to the Gauduchon conjecture (see \cite[\S IV.5]{Gauduchon84}) in complex geometry. The Gauduchon conjecture was solved by Sz\'ekelyhidi-Tosatti-Weinkove \cite{STW17}. For more references, we refer the reader to \cite{FWW10,FWW15,Popovici15,Szekelyhidi18,TW17,TW19} and references therein.

Compared to the work of Guan-Ren-Wang \cite{GRW15}, the curvature estimate in Theorem \ref{Curvature estimate} can be established without the assumption of ``strong" convexity of solution. Precisely, we prove the desired estimate for $(\eta,k)$-convex hypersurface. Clearly, the $(\eta,k)$-convexity is the natural elliptic condition for equation \eqref{cj-1}.

To obtain the existence of $(\eta, k)$-convex hypersurface satisfying the prescribed curvature equation \eqref{cj-1}, we need two additional conditions on $f$ as in  \cite{BK74,TrWe83,CNSIV}. The first condition is that there exist two positive constants $r_1 < 1 < r_2$ such that
\begin{equation}
\label{cj-6}
\begin{aligned}
f \left(X, \frac{X}{|X|}\right) \geq \,& \frac{C^k_n (n-1)^k}{r_1^k}, \ \ \mbox{ for } |X| = r_1;\\
f \left(X, \frac{X}{|X|}\right) \leq \,& \frac{C^k_n (n-1)^k}{r_2^k}, \ \ \mbox{ for } |X| = r_2,
\end{aligned}
\end{equation}
where $C_{n}^{k}=\frac{n!}{k!(n-k)!}$. The second one is that for any fixed unit vector $\nu$,
\begin{equation}
\label{cj-7}
\frac{\partial}{\partial \rho}\left(\rho^k f (X, \nu)\right) \leq 0,
\end{equation}
where $\rho=|X|$.
\begin{theorem}
\label{cj-thm-2}
Let $f \in C^2 \left((\overline{B}_{r_2} \setminus B_{r_1}) \times \mathbb{S}^n\right)$ be a positive function satisfying conditions \eqref{cj-6} and \eqref{cj-7}. Then equation \eqref{cj-1} has a unique $C^{3, \alpha}$ star-shaped $(\eta, k)$-convex solution $M$ in $\{r_1 \leq |X| \leq r_2\}$ for any $\alpha\in(0,1)$.
\end{theorem}

We now discuss the proof of Theorem \ref{Curvature estimate}. To prove the curvature estimate, we apply the maximum principle to a quantity involving the logarithm of the largest principal curvature. Since the right hand side $f$ depends on $\nu$, there are more troublesome terms when we differentiate the equation \eqref{cj-1}. We overcome this difficulty by using some properties of the operator $\sigma_{k}$. The second difficulty is how to deal with bad third order terms. Our approach is to establish the partial curvature estimate which is very useful to analyze the concavity of the operator $\sigma_{k}^{1/k}$. This gives us more good third order terms, which is enough to control the bad third order terms.

Next, we give two applications of the above idea. The first application is the $C^{2}$ estimate for the corresponding Hessian type equation in Euclidean domains.
Let $\Omega$ be a bounded domain in $\mathbb{R}^{n}$. For $\vp\in C^{2}(\Omega)$, we define
\[
\eta_{ij} = (\Delta \vp)\delta_{ij}-\vp_{ij}.
\]
The function $\vp$ is called $(\eta,k)$-convex if the eigenvalues $\lambda(\eta)$ of $\eta_{ij}$ is in $\Gamma_{k}$ for all $x\in\Omega$. We consider the following equation
\begin{equation}\label{Hessian type equation}
\sigma_{k}(\lambda(\eta)) =f(x,\vp,\nabla\vp), \quad \text{in $\Omega$},
\end{equation}
where $f$ is a positive function defined on $\ov{\Omega}\times\mathbb{R}\times\mathbb{R}^{n}$.

For equation \eqref{Hessian type equation}, when $f$ is independent of $\nabla\vp$, the $C^{2}$ estimate was proved by Caffarelli-Nirenberg-Spruck \cite{CNSIII}, where they treated a general class of fully nonlinear equations. When $f$ depends on $\nabla\vp$, equation \eqref{Hessian type equation} falls into the setup of \cite{Guan99} (see also \cite{GJ15}), and $C^{2}$ estimate was obtained under the convexity assumption of $f$ on $\nabla\vp$. In the following theorem, we remove this assumption.

\begin{theorem}\label{Global estimate}
Let $\vp\in C^{4}(\Omega)$ be a $(\eta,k)$-convex solution of \eqref{Hessian type equation}. Then there exists a constant $C$ depending only on $n$, $k$, $\|\vp\|_{C^{1}}$, $\inf f$, $\|f\|_{C^{2}}$ and $\Omega$ such that
\[
\sup_{\Omega}|\nabla^{2}\vp| \leq C \big(1 + \sup_{\de\Omega}|\nabla^{2}\vp|\big).
\]
\end{theorem}

In this paper, we omit the proof of Theorem \ref{Global estimate} since it is almost identical to that of Theorem \ref{Curvature estimate}.

The second application is an interior $C^{2}$ estimate for the following Dirichlet problem
\begin{equation}\label{Dirichlet Problem}
\begin{cases}
\sigma_{k}(\lambda(\eta)) =f(x,\vp,\nabla\vp) & \mbox {in $\Omega$}, \\
\vp = 0 & \mbox{on $\de\Omega$}.
\end{cases}
\end{equation}

\begin{theorem}\label{Interior estimate}
For the Dirichlet problem (\ref{Dirichlet Problem}), there exists a constant $C$ and $\beta$ depending only on $n$, $k$, $\|\vp\|_{C^{1}}$, $\inf f$, $\|f\|_{C^{2}}$ and $\Omega$ such that
\[
\sup_{\Omega}\left[(-\vp)^{\beta}\Delta\vp\right] \leq C.
\]
\end{theorem}

For $k$-Hessian equation, when $f$ is independent of $\nabla\vp$, the interior $C^{2}$ estimate was established by Pogorelov \cite{Pogorelov78} for $k=n$ and Chou-Wang \cite{CW01} for general $k$. When $f$ depends on $\nabla\vp$, Li-Ren-Wang \cite{LRW16} proved such estimate for $(k+1)$-convex solution (if $k=2$, the $3$-convexity condition can be replaced by $2$-convexity condition).

\bigskip

{\bf Acknowledgement.}
We thank Professor Ben Weinkove for introducing the $(n-1)$ Monge-Amp\`ere equation and many helpful comments. The work was carried out while the second author was visiting the Department of Mathematics at Northwestern University. He wishes to thank the Department and University for their hospitality. He also would like to thank China Scholarship Council for their support. The second author is supported by the National Natural Science Foundation of China (Grant Nos. 11601105, 11871243 and 11671111).

\section{Preliminaries}
Let $\mathbb{S}^{n}$ be the unit sphere in $\mathbb{R}^{n+1}$. A hypersurface $M\subset\mathbb{R}^{n+1}$ is called star-shaped if it is a radial graph of $\mathbb{S}^{n}$ for some positive function $\rho$. Thus, for $x\in\mathbb{S}^{n}$, $X(x)=\rho(x)x$ is the position vector. We have the following expressions of $g_{ij}$, $h_{ij}$ and $\nu$ (see e.g. \cite[p.1952]{GLM09})
\begin{equation}\label{Expressions of g and h}
g_{ij} = \rho^{2}\hat{g}_{ij}+\rho_{i}\rho_{j}, \quad
h_{ij} = \frac{\rho^{2}\hat{g}_{ij}+2\rho_{i}\rho_{j}-\rho\rho_{ij}}{\sqrt{\rho^{2}+|\nabla\rho|^{2}}}
\end{equation}
and
\begin{equation}\label{Expression of nu}
\nu = \frac{\rho x-\nabla\rho}{\sqrt{\rho^{2}+|\nabla\rho|^{2}}},
\end{equation}
where $\hat{g}$ and $\nabla$ denote the standard metric and the gradient on $\mathbb{S}^{n}$ respectively.

On the other hand, for $X_{0}\in M$, let $\{e_{1},e_{2},\cdots,e_{n}\}$ be a local orthonormal frame near $X_{0}$. The following formulas are well-known:
\[
\begin{split}
\text{Guass formula}: \quad & X_{ij} = -h_{ij}\nu, \\
\text{Weingarten equation}: \quad & \nu_{i} = h_{ij}e_{j}, \\
\text{Codazzi formula}: \quad & h_{ijk} = h_{ikj}, \\
\text{Guass equation}: \quad & R_{ijkl} = h_{ik}h_{jl}-h_{il}h_{jk}
\end{split}
\]
and
\begin{equation}\label{Commutation formula}
h_{ijkl} = h_{klij}+(h_{mj}h_{il}-h_{ml}h_{ij})h_{mk}+(h_{mj}h_{kl}-h_{ml}h_{kj})h_{mi}.
\end{equation}
where $R_{ijkl}$ is the curvature tensor of $M$.

\section{Curvature estimate}\label{Section curvature estimate}
In this section, we give the proof of Theorem \ref{Curvature estimate}. When $k=1$, Theorem \ref{Curvature estimate} follows from classical theory of quasi-linear PDEs. So we assume that $k\geq2$ in the following sections.

To prove Theorem \ref{Curvature estimate}, we define a function $u=\langle X,\nu\rangle$. By (\ref{Expression of nu}), it is clear that
\begin{equation}\label{Expression of u}
u = \frac{\rho^{2}}{\sqrt{\rho^{2}+|\nabla\rho|^{2}}} .
\end{equation}
Then there exists a positive constant $C$ depending on $\inf_{M}\rho$ and $\|\rho\|_{C^1}$ such that
\[
\frac{1}{C} \leq \inf_{M}u \leq u \leq  \sup_{M}u \leq C.
\]
Let $\kappa_{\mathrm{max}}$ be the largest principal curvature. From $\eta\in\Gamma_{k}\subset\Gamma_{1}$, we see that the mean curvature is positive. It suffices to prove $\kappa_{\mathrm{max}}$ is uniformly bounded from above. Without loss of generality, we may assume that the set $D=\{\kappa_{\mathrm{max}}>0\}$ is not empty. On $D$, we consider the following function
\[
Q = \log\kappa_{\mathrm{max}}-\log(u-a)+\frac{A}{2}|X|^{2},
\]
where $a=\frac{1}{2}\inf_{M}u>0$ and $A>1$ is a constant to be determined later. Note that $Q$ is continuous on $D$, and goes to $-\infty$ on $\de D$. Hence $Q$ achieves a maximum at a point $X_{0}$ with $\kappa_{\mathrm{max}}(X_{0})>0$. We choose a local orthonormal frame $\{e_{1},e_{2},\cdots,e_{n}\}$ near $X_{0}$ such that
\[
h_{ij} = \delta_{ij}h_{ii} \ \ \text{and} \ \
h_{11} \geq h_{22} \geq \cdots \geq h_{nn} \ \   \text{at $X_{0}$}.
\]
Recalling that $\eta_{ii}=\sum_{k\neq i}h_{kk}$, we have
\[
\eta_{11} \leq \eta_{22} \leq \cdots \leq \eta_{nn}.
\]
Near $X_{0}$, we define a new function $\hat{Q}$ by
\[
\hat{Q} = \log h_{11}-\log(u-a)+\frac{A}{2}|X|^{2}.
\]
Since $h_{11}(X_{0})=\kappa_{\mathrm{max}}(X_{0})$ and $h_{11}\leq\kappa_{\mathrm{max}}$ near $X_{0}$, $\hat{Q}$ achieves a maximum at $X_{0}$. From now on, all the calculations will be carried out at $X_{0}$. For convenience, we introduce the following notations:
\[
G(\eta) = \sigma_{k}^{\frac{1}{k}}(\eta), \quad
G^{ij} = \frac{\de G}{\de\eta_{ij}}, \quad
G^{ij,kl} = \frac{\de^{2} G}{\de\eta_{ij}\de\eta_{kl}}, \quad
F^{ii} = \sum_{k\neq i}G^{kk}.
\]
Thus,
\[
G^{ii} = \frac{1}{k}[\sigma_{k}(\eta)]^{\frac{1}{k}-1}\sigma_{k-1}(\eta|i),
\]
where $\sigma_{k-1}(\eta|i)$ denotes $(k-1)$-th elementary symmetric function with $\eta_{ii}=0$. It then follows that
\[
G^{11} \geq G^{22} \geq \cdots \geq G^{nn}, \quad F^{11} \leq F^{22}  \leq \cdots \leq F^{nn}.
\]
Applying the maximum principle, for any $1\leq i\leq n$, we have
\begin{equation}\label{hat Q i}
0 = \hat{Q}_{i}
= \frac{h_{11i}}{h_{11}}-\frac{u_{i}}{u-a}+A\langle X,e_{i}\rangle
\end{equation}
and
\begin{equation}\label{F ii hat Q ii}
0 \geq F^{ii}\hat{Q}_{ii}
= F^{ii}(\log h_{11})_{ii}-F^{ii}(\log(u-a))_{ii}+\frac{A}{2}F^{ii}(|X|^{2})_{ii}.
\end{equation}

We first need to estimate each term in (\ref{F ii hat Q ii}).

\begin{lemma}\label{Calculation}
We have
\[
\begin{split}
0 \geq {} & -\frac{2}{h_{11}}\sum_{i\geq2}G^{1i,i1}h_{11i}^{2}-\frac{F^{ii}h_{11i}^{2}}{h_{11}^{2}} \\
& +\frac{aF^{ii}h_{ii}^{2}}{u-a}+\frac{F^{ii}u_{i}^{2}}{(u-a)^{2}}
-Ch_{11}+A\sum_{i}F^{ii}-CA.
\end{split}
\]
\end{lemma}

\begin{proof}
We first deal with the term $\frac{A}{2}F^{ii}(|X|^{2})_{ii}$ in (\ref{F ii hat Q ii}). Since $\eta_{ii}=\sum_{j\neq i}h_{jj}$, we have
\[
\sum_{i}\eta_{ii} = (n-1)\sum_{i}h_{ii}, \quad
h_{ii} = \frac{1}{n-1}\sum_{k}\eta_{kk}-\eta_{ii}.
\]
It then follows that
\begin{equation}\label{Calculation eqn 1}
\begin{split}
\sum_{i}F^{ii}h_{ii}
= {} & \sum_{i}\left(\sum_{k}G^{kk}-G^{ii}\right)\left(\frac{1}{n-1}\sum_{l}\eta_{ll}-\eta_{ii}\right) \\
= {} & \sum_{i}G^{ii}\eta_{ii}
= \frac{1}{k}[\sigma_{k}(\eta)]^{\frac{1}{k}-1}\sum_{i}\eta_{ii}\sigma_{k-1}(\eta|i)
= f^{\frac{1}{k}}.
\end{split}
\end{equation}
Combining this with Gauss formula, we obtain
\begin{equation}\label{Calculation eqn 2}
\begin{split}
\frac{A}{2}F^{ii}(|X|^{2})_{ii}
= {} & A\sum_{i}F^{ii}(1+\langle X,X_{ii}\rangle) \\
= {} & A\sum_{i}F^{ii}(1-h_{ii}\langle X,\nu\rangle) \\
= {} & A\sum_{i}F^{ii}-Auf^{\frac{1}{k}}.
\end{split}
\end{equation}

For the term $-F^{ii}(\log(u-a))_{ii}$ in (\ref{F ii hat Q ii}), we compute
\[
-F^{ii}(\log(u-a))_{ii}
= -\frac{F^{ii}u_{ii}}{u-a}+\frac{F^{ii}u_{i}^{2}}{(u-a)^{2}}.
\]
Using Guass formula, Weingarten equation and Codazzi formula,
\[
u_{i} =  h_{ii}\langle X,e_{i}\rangle, \quad
u_{ii} = \sum_{k}h_{iik}\langle X,e_{k}\rangle-uh_{ii}^{2}+h_{ii}.
\]
It then follows that
\begin{equation}\label{Calculation eqn 7}
\begin{split}
& -F^{ii}(\log(u-a))_{ii} \\
= {} & -\frac{1}{u-a}\sum_{k}F^{ii}h_{iik}\langle X,e_{k}\rangle+\frac{uF^{ii}h_{ii}^{2}}{u-a}
-\frac{F^{ii}h_{ii}}{u-a}+\frac{F^{ii}u_{i}^{2}}{(u-a)^{2}} \\
= {} & -\frac{1}{u-a}\sum_{k}F^{ii}h_{iik}\langle X,e_{k}\rangle+\frac{uF^{ii}h_{ii}^{2}}{u-a}
-\frac{f^{\frac{1}{k}}}{u-a}+\frac{F^{ii}u_{i}^{2}}{(u-a)^{2}},
\end{split}
\end{equation}
where we used (\ref{Calculation eqn 1}) in the last line. By the definitions of $F^{ii}$ and $\eta_{ii}$, we have
\begin{equation}\label{Calculation eqn 3}
\begin{split}
F^{ii}h_{iik} = {} &  \left(\sum_{j}G^{jj}-G^{ii}\right)h_{iik} \\
= {} & \left(\sum_{i}G^{ii}\right)H_{k}-\sum_{i}G^{ii}h_{iik} = G^{ii}\eta_{iik}.
\end{split}
\end{equation}
On the other hand, the curvature equation \eqref{cj-1} can be written as
\begin{equation}\label{Curvature equation 2}
G(\eta) = \ti{f},
\end{equation}
where $\ti{f}=f^{\frac{1}{k}}$. Differentiating (\ref{Curvature equation 2}), we obtain
\[
G^{ii}\eta_{iik} = (d_{X}\ti{f})(e_{k})+h_{kk}(d_{\nu}\ti{f})(e_{k}).
\]
Then (\ref{Calculation eqn 3}) gives
\begin{equation}\label{C1 eqn 3}
F^{ii}h_{iik} = (d_{X}\ti{f})(e_{k})+h_{kk}(d_{\nu}\ti{f})(e_{k}).
\end{equation}
It then follows that
\[
-\frac{1}{u-a}\sum_{k}F^{ii}h_{iik}\langle X,e_{k}\rangle
\geq  -\frac{1}{u-a}\sum_{k}h_{kk}(d_{\nu}\ti{f})(e_{k})\langle X,e_{k}\rangle-C.
\]
Substituting this into (\ref{Calculation eqn 7}), we have
\begin{equation}\label{Calculation eqn 4}
\begin{split}
-F^{ii}(\log(u-a))_{ii}
\geq {} & -\frac{1}{u-a}\sum_{k}h_{kk}(d_{\nu}\ti{f})(e_{k})\langle X,e_{k}\rangle \\
& +\frac{uF^{ii}h_{ii}^{2}}{u-a}+\frac{F^{ii}u_{i}^{2}}{(u-a)^{2}}-C.
\end{split}
\end{equation}

For the term $F^{ii}(\log h_{11})_{ii}$ in (\ref{F ii hat Q ii}), we compute
\begin{equation}\label{Calculation eqn 6}
F^{ii}(\log h_{11})_{ii} = \frac{F^{ii}h_{11ii}}{h_{11}}-\frac{F^{ii}h_{11i}^{2}}{h_{11}^{2}}.
\end{equation}
By (\ref{Commutation formula}) and (\ref{Calculation eqn 1}), we have
\[
\begin{split}
F^{ii}h_{11ii}
= {} & F^{ii}h_{ii11}+F^{ii}(h_{i1}^{2}-h_{ii}h_{11})h_{ii}+F^{ii}(h_{ii}h_{11}-h_{i1}^{2})h_{11} \\
= {} & F^{ii}h_{ii11}-F^{ii}h_{ii}^{2}h_{11}+F^{ii}h_{ii}h_{11}^{2} \\
= {} & F^{ii}h_{ii11}-F^{ii}h_{ii}^{2}h_{11}+f^{\frac{1}{k}}h_{11}^{2}.
\end{split}
\]
Differentiating (\ref{Curvature equation 2}) twice and using the similar argument of (\ref{Calculation eqn 3}), we obtain
\[
F^{ii}h_{ii11} = G^{ii} \eta_{ii11} \geq -G^{ij,kl}\eta_{ij1}\eta_{kl1}+\sum_{k}h_{k11}(d_{\nu}\ti{f})(e_{k})-Ch_{11}^{2}-C.
\]
Applying the concavity of $G$ and Codazzi formula, we have
\[
-G^{ij,kl}\eta_{ij1}\eta_{kl1}
\geq -2\sum_{i\geq2}G^{1i,i1}\eta_{1i1}^{2}
= -2\sum_{i\geq2}G^{1i,i1}h_{1i1}^{2}
= -2\sum_{i\geq2}G^{1i,i1}h_{11i}^{2}.
\]
It then follows that
\[
F^{ii}h_{11ii}
\geq -2\sum_{i\geq2}G^{1i,i1}h_{11i}^{2}+\sum_{k}h_{k11}(d_{\nu}\ti{f})(e_{k})-F^{ii}h_{ii}^{2}h_{11}-Ch_{11}^{2}-C.
\]
Substituting this into (\ref{Calculation eqn 6}),
\begin{equation}\label{Calculation eqn 5}
\begin{split}
F^{ii}(\log h_{11})_{ii}
\geq {} & -\frac{2}{h_{11}}\sum_{i\geq2}G^{1i,i1}h_{11i}^{2}+\frac{1}{h_{11}}\sum_{k}h_{k11}(d_{\nu}\ti{f})(e_{k}) \\
& -\frac{F^{ii}h_{11i}^{2}}{h_{11}^{2}}-F^{ii}h_{ii}^{2}-Ch_{11}.
\end{split}
\end{equation}
Combining (\ref{F ii hat Q ii}), (\ref{Calculation eqn 2}), (\ref{Calculation eqn 4}) and (\ref{Calculation eqn 5}), we obtain
\[
\begin{split}
0 \geq F^{ii}\hat{Q}_{ii}
\geq {} & -\frac{2}{h_{11}}\sum_{i\geq2}G^{1i,i1}h_{11i}^{2}-\frac{F^{ii}h_{11i}^{2}}{h_{11}^{2}} \\
& +\frac{1}{h_{11}}\sum_{k}h_{k11}(d_{\nu}\ti{f})(e_{k})-\frac{1}{u-a}\sum_{k}h_{kk}(d_{\nu}\ti{f})(e_{k})\langle X,e_{k}\rangle \\
& +\frac{aF^{ii}h_{ii}^{2}}{u-a}+\frac{F^{ii}u_{i}^{2}}{(u-a)^{2}}
-Ch_{11}+A\sum_{i}F^{ii}-CA.
\end{split}
\]
By Codazzi formula, $u_{k}=h_{kk}\langle X,e_{k}\rangle$ and (\ref{hat Q i}), we have
\[
\begin{split}
& \frac{1}{h_{11}}\sum_{k}h_{k11}(d_{\nu}\ti{f})(e_{k})-\frac{1}{u-a}\sum_{k}h_{kk}(d_{\nu}\ti{f})(e_{k})\langle e_{k},X\rangle \\
= {} & \sum_{k}\left(\frac{h_{11k}}{h_{11}}-\frac{u_{k}}{u-a}\right)(d_{\nu}\ti{f})(e_{k})
= -A\sum_{k}(d_{\nu}\ti{f})(e_{k})\langle X,e_{k}\rangle \geq -CA.
\end{split}
\]
Therefore,
\[
\begin{split}
0 \geq {} & -\frac{2}{h_{11}}\sum_{i\geq2}G^{1i,i1}h_{11i}^{2}-\frac{F^{ii}h_{11i}^{2}}{h_{11}^{2}} \\
& +\frac{aF^{ii}h_{ii}^{2}}{u-a}+\frac{F^{ii}u_{i}^{2}}{(u-a)^{2}}
-Ch_{11}+A\sum_{i}F^{ii}-CA,
\end{split}
\]
as required.
\end{proof}

Next, we deal with the bad term $-Ch_{11}$.

\begin{lemma}\label{Bad term}
If $h_{11}\geq CA$ and $A\geq C$ for some uniform constant $C$, then we have
\[
\begin{split}
Ch_{11} \leq \frac{aF^{ii}h_{ii}^{2}}{2(u-a)}+\frac{A}{2}\sum_{i}F^{ii}.
\end{split}
\]
\end{lemma}

\begin{proof}
The proof splits into two cases. The positive constant $\delta$ will be determined later.
\bigskip

\noindent
{\bf Case 1.} \ $|h_{ii}| \leq \delta h_{11}$ for all $i\geq2$.

\bigskip

In this case, we have
\[
|\eta_{11}|  \leq (n-1)\delta h_{11}
\]
and
\[
[1-(n-2)\delta]h_{11} \leq \eta_{22}
\leq \cdots \leq \eta_{nn} \leq [1+(n-2)\delta]h_{11}.
\]
It then follows that
\[
\sigma_{k-1}(\eta) = \sigma_{k-1}(\eta|1)+\eta_{11}\sigma_{k-2}(\eta|1)
\geq (1-C\delta)h_{11}^{k-1}-C\delta h_{11}^{k-1}.
\]
Choosing $\delta$ sufficiently small and using $k\geq2$,
\begin{equation}\label{Case 1 eqn 1}
\sigma_{k-1}(\eta) \geq \frac{h_{11}^{k-1}}{2} \geq \frac{h_{11}}{2}.
\end{equation}
By the definition of $G^{ii}$ and $F^{ii}$, we obtain
\begin{equation}\label{Case 1 eqn 2}
\sum_{i}F^{ii} = (n-1)\sum_{i}G^{ii}
= \frac{(n-1)(n-k+1)}{k}[\sigma_{k}(\eta)]^{\frac{1}{k}-1}\sigma_{k-1}(\eta).
\end{equation}
Thanks to $\sigma_{k}(\eta)=f$, we have
\[
\sum_{i}F^{ii}
=  \frac{(n-1)(n-k+1)}{k}f^{\frac{1}{k}-1}\sigma_{k-1}(\eta)
\geq \frac{\sigma_{k-1}(\eta)}{C}.
\]
Combining this with (\ref{Case 1 eqn 1}), and choosing $A$ sufficiently large, we obtain
\[
Ch_{11} \leq C\sigma_{k-1}(\eta) \leq \frac{A}{2}\sum_{i}F^{ii},
\]
as required.

\bigskip

\noindent
{\bf Case 2.} \ $h_{22}>\delta h_{11}$ or $h_{nn}<-\delta h_{11}$ .

\bigskip

In this case, we have
\[
\frac{aF^{ii}h_{ii}^{2}}{2(u-a)}
\geq \frac{1}{C}(F^{22}h_{22}^{2}+F^{nn}h_{nn}^{2})
\geq \frac{\delta^{2}}{C}F^{22}h_{11}^{2}.
\]
By the definitions of $F^{ii}$ and $G^{ii}$,
\begin{equation}\label{Case 2 eqn 1}
F^{22} = \sum_{i\neq 2}G^{ii} \geq G^{11} \geq \frac{1}{n}\sum_{i}G^{ii} = \frac{1}{n(n-1)}\sum_{i}F^{ii}.
\end{equation}
It then follows that
\begin{equation}\label{Case 2 eqn 2}
\frac{aF^{ii}h_{ii}^{2}}{2(u-a)}
\geq \frac{\delta^{2}h_{11}^{2}}{C}\sum_{i}F^{ii}.
\end{equation}
Using (\ref{Case 1 eqn 2}) and Maclaurin's inequality,
\begin{equation}\label{Case 2 eqn 3}
\sum_{i}F^{ii} =  \frac{(n-1)(n-k+1)}{k}[\sigma_{k}(\eta)]^{\frac{1}{k}-1}\sigma_{k-1}(\eta)
\geq \frac{1}{C}.
\end{equation}
Combining this with (\ref{Case 2 eqn 2}), if $h_{11}\geq\frac{C}{\delta^{2}}$, we obtain
\[
Ch_{11} \leq \frac{\delta^{2}h_{11}^{2}}{C}\sum_{i}F^{ii}
\leq \frac{aF^{ii}h_{ii}^{2}}{2(u-a)},
\]
as required.
\end{proof}

Combining Lemma \ref{Calculation} and \ref{Bad term}, we obtain
\begin{equation}\label{Inequality}
\begin{split}
0 \geq {} & -\frac{2}{h_{11}}\sum_{i\geq 2}G^{1i,i1}h_{11i}^{2}-\frac{F^{ii}h_{11i}^{2}}{h_{11}^{2}} \\
& +\frac{aF^{ii}h_{ii}^{2}}{2(u-a)}+\frac{F^{ii}u_{i}^{2}}{(u-a)^{2}}+\frac{A}{2}\sum_{i}F^{ii}-CA.
\end{split}
\end{equation}

The following lemma can be regarded as the partial curvature estimate.

\begin{lemma}\label{Partial curvature estimate}
If $h_{11}\geq CA$ and $A\geq C$ for some uniform constant $C$, then we have
\[
|h_{ii}| \leq CA, \quad \text{for $i\geq 2$}.
\]
\end{lemma}

\begin{proof}
Using (\ref{hat Q i}) and the Cauchy-Schwarz inequality,
\[
\frac{F^{ii}h_{11i}^{2}}{h_{11}^{2}}
\leq \frac{1+\ve}{(u-a)^{2}}F^{ii}u_{i}^{2}+\left(1+\frac{1}{\ve}\right)A^{2}F^{ii}\langle X,e_{i}\rangle^{2}.
\]
Substituting this into (\ref{Inequality}) and dropping the positive term $-\frac{2}{h_{11}}\sum_{i\geq2}G^{1i,i1}h_{11i}^{2}$, we have
\[
\begin{split}
0 \geq {} & \frac{aF^{ii}h_{ii}^{2}}{2(u-a)}-\frac{\ve F^{ii}u_{i}^{2}}{(u-a)^{2}}-\frac{CA^{2}}{\ve}\sum_{i}F^{ii}-CA \\
\geq {} & \left(\frac{a}{2(u-a)}-\frac{C\ve}{(u-a)^{2}}\right)F^{ii}h_{ii}^{2}-\frac{CA^{2}}{\ve}\sum_{i}F^{ii}-CA,
\end{split}
\]
where we used $u_{i}=h_{ii}\langle X,e_{i}\rangle$ in the second inequality. Choosing $\ve$ sufficiently small and recalling (\ref{Case 2 eqn 3}), we obtain
\begin{equation}\label{Partial curvature estimate eqn 1}
0 \geq \frac{F^{ii}h_{ii}^{2}}{C}-\frac{CA^{2}}{\ve}\sum_{i}F^{ii}.
\end{equation}
By (\ref{Case 2 eqn 1}), for $i\geq2$, we have
\[
F^{ii} \geq F^{22} \geq \frac{1}{n(n-1)}\sum_{k}F^{kk}.
\]
Then (\ref{Partial curvature estimate eqn 1}) gives
\[
0 \geq \frac{1}{C}\left(\sum_{i}F^{ii}\right)\left(\sum_{k\geq2}h_{kk}^{2}\right)-\frac{CA^{2}}{\ve}\sum_{i}F^{ii}.
\]
which implies
\[
\sum_{k\geq2}h_{kk}^{2} \leq CA^{2},
\]
as required.
\end{proof}

Now we are in a position to prove Theorem \ref{Curvature estimate}.

\begin{proof}[Proof of Theorem \ref{Curvature estimate}]
By (\ref{hat Q i}) and the Cauchy-Schwarz inequality, we have
\[
\frac{F^{11}h_{111}^{2}}{h_{11}^{2}}
\leq \frac{1+\ve}{(u-a)^{2}}F^{11}u_{1}^{2}+\left(1+\frac{1}{\ve}\right)A^{2}F^{11}\langle X,e_{1}\rangle^{2}.
\]
Recalling $u_{1}=h_{11}\langle\de_{1},X\rangle$ and choosing $\ve$ sufficiently small,
\[
\begin{split}
\frac{F^{11}h_{111}^{2}}{h_{11}^{2}}
\leq {} &\frac{F^{11}u_{1}^{2}}{(u-a)^{2}}+\frac{\ve F^{11}u_{1}^{2}}{(u-a)^{2}}+\frac{CA^{2}F^{11}}{\ve} \\
\leq {} & \frac{F^{11}u_{1}^{2}}{(u-a)^{2}}+\frac{C\ve F^{11}h_{11}^{2}}{(u-a)^{2}}+\frac{CA^{2}F^{11}}{\ve} \\
\leq {} & \frac{F^{11}u_{1}^{2}}{(u-a)^{2}}+\frac{aF^{ii}h_{ii}^{2}}{16(u-a)}+\frac{CA^{2}F^{11}}{\ve}.
\end{split}
\]
Without loss of generality, we assume that $h_{11}^{2}\geq\frac{CA^{2}}{\ve}$, which implies
\[
\frac{CA^{2}F^{11}}{\ve} \leq \frac{aF^{ii}h_{ii}^{2}}{16(u-a)}.
\]
It then follows that
\begin{equation}\label{Curvature estimate eqn 1}
\frac{F^{11}h_{111}^{2}}{h_{11}^{2}}
\leq \frac{F^{11}u_{1}^{2}}{(u-a)^{2}}+\frac{aF^{ii}h_{ii}^{2}}{8(u-a)}.
\end{equation}

Thanks to Lemma \ref{Partial curvature estimate}, we assume that $|h_{ii}|\leq\delta h_{11}$ for $i\geq2$, where $\delta$ is a constant to be determined later. Thus,
\[
\frac{1}{h_{11}} \leq \frac{1+\delta}{h_{11}-h_{ii}}.
\]
Combining this with $-G^{1i,i1}=\frac{G^{11}-G^{ii}}{\eta_{ii}-\eta_{11}}$ for $i\geq2$, we obtain
\begin{equation}\label{Curvature estimate eqn 2}
\begin{split}
\sum_{i\geq2}\frac{F^{ii}h_{11i}^{2}}{h_{11}^{2}}
= {} & \sum_{i\geq2}\frac{F^{ii}-F^{11}}{h_{11}^{2}}h_{11i}^{2}+\sum_{i\geq2}\frac{F^{11}h_{11i}^{2}}{h_{11}^{2}} \\
\leq {} & \frac{1+\delta}{h_{11}}\sum_{i\geq2}\frac{F^{ii}-F^{11}}{h_{11}-h_{ii}}h_{11i}^{2}+\sum_{i\geq2}\frac{F^{11}h_{11i}^{2}}{h_{11}^{2}} \\
= {} & \frac{1+\delta}{h_{11}}\sum_{i\geq2}\frac{G^{11}-G^{ii}}{\eta_{ii}-\eta_{11}}h_{11i}^{2}
+\sum_{i\geq2}\frac{F^{11}h_{11i}^{2}}{h_{11}^{2}} \\
= {} & -\frac{1+\delta}{h_{11}}\sum_{i\geq2}G^{1i,i1}h_{11i}^{2}
+\sum_{i\geq2}\frac{F^{11}h_{11i}^{2}}{h_{11}^{2}}.
\end{split}
\end{equation}
Using (\ref{hat Q i}), the Cauchy-Schwarz inequality and $u_{i}=h_{ii}\langle X,e_{i}\rangle$, we have
\[
\begin{split}
\sum_{i\geq2}\frac{F^{11}h_{11i}^{2}}{h_{11}^{2}}
\leq {} & 2\sum_{i\geq2}\frac{F^{11}u_{i}^{2}}{(u-a)^{2}}+2A^{2}\sum_{i\geq2}F^{11}\langle X,e_{i}\rangle^{2} \\
\leq {} & C\sum_{i\geq2}F^{11}h_{ii}^{2}+CA^{2}F^{11}.
\end{split}
\]
Recalling that $|h_{ii}|\leq\delta h_{11}$ for $i\geq2$ and choosing $\delta$ sufficiently small, we see that
\[
C\sum_{i\geq2}F^{11}h_{ii}^{2}+CA^{2}F^{11}
\leq \frac{aF^{11}h_{11}^{2}}{8(u-a)}.
\]
It then follows that
\[
\sum_{i\geq2}\frac{F^{11}h_{11i}^{2}}{h_{11}^{2}}
\leq \frac{aF^{11}h_{11}^{2}}{8(u-a)}.
\]
Combining this with (\ref{Curvature estimate eqn 2}),
\begin{equation}\label{Curvature estimate eqn 3}
\begin{split}
\sum_{i\geq2}\frac{F^{ii}h_{11i}^{2}}{h_{11}^{2}}
\leq {} & -\frac{1+\delta}{h_{11}}\sum_{i\geq2}G^{1i,i1}h_{11i}^{2}+\frac{aF^{11}h_{11}^{2}}{8(u-a)} \\
\leq {} & -\frac{2}{h_{11}}\sum_{i\geq2}G^{1i,i1}h_{11i}^{2}+\frac{aF^{11}h_{11}^{2}}{8(u-a)}.
\end{split}
\end{equation}
Substituting (\ref{Curvature estimate eqn 1}) and (\ref{Curvature estimate eqn 3}) into (\ref{Inequality}), we obtain
\begin{equation}\label{Curvature estimate eqn 4}
\begin{split}
0 \geq {} & \frac{aF^{ii}h_{ii}^{2}}{4(u-a)}+\frac{A}{2}\sum_{i}F^{ii}-CA \\
= {} & \frac{aF^{ii}h_{ii}^{2}}{4(u-a)}+\frac{(n-1)A}{2}\sum_{i}G^{ii}-CA.
\end{split}
\end{equation}
It then follows that
\[
\sum_{i}G^{ii} \leq C.
\]
Combining this with \cite[Lemma 2.2]{HMW10}, we obtain
\[
F^{11} = \sum_{i\neq 1}G^{ii} \geq \frac{1}{C}.
\]
Substituting this into (\ref{Curvature estimate eqn 4}), we obtain
\[
h_{11} \leq C,
\]
as required.
\end{proof}

\section{Proof of Theorem \ref{cj-thm-2}}
In this section, we give the proof of Theorem \ref{cj-thm-2}. Since $M$ is star-shaped, we assume that $M$ is the radial graph of positive function $\rho$ on $\mathbb{S}^{n}$. So $X (x) = \rho (x) x$ is the position vector of $M$. We first prove the gradient estimate.

\begin{lemma}
\label{cj-lem2}
Suppose that $f$ satisfies \eqref{cj-7} and $\rho$ has positive lower and upper bounds. Then there exists a constant $C$ depending only on $n$, $k$, $\inf_{M}\rho$, $\sup_{M}\rho$, $\inf f$ and $\|f\|_{C^{1}}$ such that
\begin{equation}
\label{cj-lem2-1}
|\nabla \rho| \leq C,
\end{equation}
where $\nabla$ denotes the gradient on $\mathbb{S}^{n}$.
\end{lemma}
\begin{proof}
By \eqref{Expression of u}, it suffices to obtain a positive lower bound of $u$. As in \cite{GLM09}, we consider the following quantity
\[
w = - \log u + \gamma (|X|^2),
\]
where $\gamma$ is a single-variable function to be determined later. Suppose that the maximum of $w$ is achieved at $X_0 \in M$. If the vector $X_0$ is not parallel to the outer normal vector of $M$ at $X_0$, we can choose the local orthonormal frame $\{e_1,e_2,\cdots, e_n\}$ near $X_0$ such that
\begin{equation}\label{C1 eqn 2}
\langle X, e_1\rangle \neq 0 \ \, \text{and} \ \, \langle X, e_i\rangle = 0 \ \, \text{for $i \geq 2$}.
\end{equation}
Using Weingarten equation, we obtain
\[
u_{i} = \sum_{j}h_{ij}\langle X,e_{j}\rangle = h_{i1}\langle X,e_{1}\rangle.
\]
Therefore, at $X_0$, we have
\begin{equation}
\label{C1 eqn 4}
0 = w_i = - \frac{u_i}{u} + 2 \gamma' \langle X, e_i\rangle
= -\frac{h_{i1}\langle X,e_{1}\rangle}{u} + 2 \gamma' \langle X, e_i\rangle,
\end{equation}
so that $h_{11} = 2 \gamma' u$ and $h_{1i} = 0$ for $i \geq 2$. Furthermore, after rotating $\{e_2,e_3,\cdots, e_n\}$, we assume that $(h_{ij} (X_0))$ is diagonal. For convenience, we use the same notations as in Section \ref{Section curvature estimate}. By \eqref{C1 eqn 4}, we obtain $\frac{u_{i}^{2}}{u^{2}}=4(\gamma')^{2}\langle X,e_{i}\rangle^{2}$. Using the maximum principle,
\begin{equation}
\label{cj-9}
\begin{split}
0 \geq {} & F^{ii} w_{ii} \\
= {} & F^{ii} \left(-\frac{u_{ii}}{u} + \frac{u_i^2}{u^2} + \gamma''(|X|^{2})_{i}^2 + \gamma' (|X|^2)_{ii}\right) \\
= {} & -\frac{F^{ii}u_{ii}}{u}+4(\gamma')^{2}F^{ii}\langle X,e_{i}\rangle^{2}
+ 4 \gamma'' F^{ii}\langle X, e_i\rangle^2 + \gamma'F^{ii}(|X|^2)_{ii} \\
= {} & -\frac{F^{ii}u_{ii}}{u}+4 (\gamma'^2 + \gamma'') F^{11} \langle X, e_1\rangle^2 + \gamma' F^{ii} (|X|^{2})_{ii}.
\end{split}
\end{equation}
To deal with terms $-\frac{F^{ii}u_{ii}}{u}$ and $\gamma' F^{ii} (|X|^{2})_{ii}$, we apply the similar argument of Lemma \ref{Calculation} and obtain
\[
F^{ii}u_{ii} = \langle X,e_{1}\rangle(d_{X}\ti{f})(e_{1})+h_{11}\langle X,e_{1}\rangle(d_{\nu}\ti{f})(e_{1})-uF^{ii}h_{ii}^{2}+\ti{f}
\]
and
\[
F^{ii} (|X|^{2})_{ii} = 2\sum_{i}F^{ii}-2u\ti{f},
\]
where $\ti{f}=f^{\frac{1}{k}}$. Substituting these into \eqref{cj-9} and using $h_{11} = 2 \gamma' u$,
\begin{equation}
\label{cj-11}
\begin{split}
0 \geq {} & - \frac{1}{u} \left(\langle X, e_1\rangle (d_X \ti{f}) (e_1) + \ti{f}\right) - 2 \gamma'\langle X, e_1\rangle(d_{\nu}\ti{f})(e_{1})  \\
   & + F^{ii} h_{ii}^2 + 4 (\gamma'^2 + \gamma'') F^{11} \langle X, e_1\rangle^2 + 2 \gamma' \sum_{i} F^{ii} - 2 u\gamma' \ti{f}.
\end{split}
\end{equation}
By \eqref{C1 eqn 2}, at $X_{0}$, we have
\[
X = \langle X, e_1\rangle e_1 + \langle X, \nu\rangle \nu,
\]
which implies
\begin{equation}
\label{cj-12}
(d_X \ti{f}) (X) = \langle X, e_1\rangle (d_X \ti{f}) (e_1) + \langle X, \nu\rangle (d_X \ti{f}) (\nu).
\end{equation}
From \eqref{cj-7}, \eqref{cj-12} and $X(x)=\rho(x) x$, we see that
\[
\begin{split}
0 \geq {} &  \frac{\partial}{\partial \rho} \left(\rho^k f (X, \nu)\right) = \frac{\partial}{\partial \rho} \left(\rho^k \ti{f}^{k} (X, \nu)\right) \\
= {} & k(\rho\ti{f})^{k-1}\left(\ti{f}+(d_{X}\ti{f})(X) \right) \\
= {} & k(\rho\ti{f})^{k-1}\left(\ti{f}+ \langle X, e_1\rangle (d_X \ti{f}) (e_1) + \langle X, \nu\rangle (d_X \ti{f}) (\nu)\right).
\end{split}
\]
It then follows that
\[
-\left(\langle X, e_1\rangle (d_X \ti{f}) (e_1) + \ti{f}\right)
\geq  \langle X, \nu\rangle (d_X \ti{f}) (\nu) = u (d_X \ti{f}) (\nu).
\]
Substituting this into \eqref{cj-11}, we obtain
\begin{equation}
\label{cj-14}
\begin{split}
0 \geq {} & (d_X \ti{f}) (\nu) - 2 \gamma'\langle X, e_1\rangle (d_{\nu}\ti{f})(e_{1}) \\
   & + F^{ii} h_{ii}^2 + 4 (\gamma'^2 + \gamma'') F^{11} \langle X, e_1\rangle^2 + 2 \gamma' \sum_{i} F^{ii} - 2 u \gamma'\ti{f}.
\end{split}
\end{equation}
It is clear that
\[
|X|^{2} = \rho^{2} \geq \inf_{M}\rho^{2} > 0.
\]
Without loss of generality, we assume that
\begin{equation}\label{C1 eqn 5}
\langle X,e_{1}\rangle^{2} \geq \frac{1}{2}\inf_{M}\rho^{2}.
\end{equation}
Otherwise we obtain the positive lower bound of $u$ at $X_{0}$ directly:
\[
u^2 = \langle X, \nu\rangle^2
= |X|^2 - \langle X, e_1\rangle^2 \geq \frac{1}{2}\inf_{M}\rho^{2}.
\]
Now we choose $\gamma (t) = \frac{\alpha}{t}$, where $\alpha$ is a constant to be determined later. Recalling that $h_{11}=2\gamma'u$ at $X_{0}$, we have $h_{11}(X_{0})< 0$. Using $H > 0$, there exists $2 \leq k \leq n$ such that $h_{kk} (X_0) > 0$. Combining this with the definitions of $\eta_{ii}$ and $G^{ii}$,
\[
\eta_{ii}< \eta_{11} \  \text{and} \ G^{ii}\geq G^{11}.
\]
It then follows that
\begin{equation}
\label{cj-15}
F^{11} = \sum_{j \neq 1} G^{jj} \geq \frac{1}{2} \sum_{i} G^{ii} = \frac{1}{2(n-1)}\sum_{i}F^{ii}.
\end{equation}
Combining \eqref{cj-14},  \eqref{C1 eqn 5} and \eqref{cj-15},
\[
0 \geq \left(\frac{\alpha^{2}}{C}-C\alpha\right)\sum_{i}F^{ii}-C\alpha,
\]
where $C$ is a constant depending only on $n$, $k$, $\inf_{M}\rho$, $\sup_{M}\rho$ and $\|\ti{f}\|_{C^{1}}$. Using (\ref{Case 2 eqn 3}) and choosing $\alpha$ sufficiently large, we obtain a contradiction. This shows that $X$ is parallel to $\nu$ at $X_0$. Hence, at $X_{0}$, we obtain
\[
u = \langle X,\nu \rangle = \rho \geq \inf_{M}\rho,
\]
as required.
\end{proof}

Now we use the continuity method as in \cite{CNSIV} to prove Theorem \ref{cj-thm-2}.

\begin{proof}[Proof of Theorem \ref{cj-thm-2}]
For $t\in[0,1]$, we consider the following family of functions
\[
f^t (X, \nu) = t f (X, \nu) + (1-t) C_n^k (n-1)^k \left[\frac{1}{|X|^k} + \ve \left(\frac{1}{|X|^k} - 1\right)\right],
\]
where the constant $\ve$ is small sufficiently such that
\[
\min_{r_1 \leq \rho \leq r_2}\left[\frac{1}{\rho^k} + \ve \left(\frac{1}{\rho^k} - 1\right)\right] \geq c_0 > 0
\]
for some positive constant $c_0$. It is easy to see that $f^t (X, \nu)$ satisfies \eqref{cj-6} and \eqref{cj-7} with strict inequalities for
$0 \leq t < 1$.

Let $M_{t}$ be the solution of the equation
\[
\sigma_{k}(\lambda(\eta)) = f^{t}(X_{t},\nu_{t}),
\]
where $X_{t}$ and $\nu_{t}$ are position vector and unit outer normal of $M_{t}$ respectively. Clearly, when $t=0$, we have $M_{0}=\mathbb{S}^{n}$ and $X_{0}=x$. For $t\in(0,1)$, suppose that $x_{0}$ is the maximum point of the function $\rho_{t}=|X_{t}|$. Thus at $x_{0}$, by (\ref{Expressions of g and h}), we have
\[
g_{ij} = \rho^2_t \hat{g}_{ij}, \quad  h_{ij} = -(\rho_t)_{ij} + \rho_t\hat{g}_{ij}.
\]
It then follows that
\[
\eta_{ij} = Hg_{ij} - h_{ij} \geq (n-1) \rho_t \hat{g}_{ij},
\]
which implies
\[
\sigma_{k}(\lambda(\eta)) \geq \sigma_k \left(\frac{n-1}{\rho_t} (1, \cdots, 1)\right) = \frac{C_n^k (n-1)^k}{\rho_{t}^{k}}.
\]
On the other hand, at $x_{0}$, the unit outer normal $\nu_t$ is parallel to $X_t$, i.e., $\nu_{t} = \frac{X_{t}}{|X_{t}|}$. If $\rho_{t}(x_{0})=r_{2}$, then we obtain
\[
\frac{C_n^k (n-1)^k}{r_2^k} > f \left(X_t, \frac{X_t}{|X_t|}\right) = f (X_t, \nu_t) = \sigma_k (\lambda(\eta))
\geq  \frac{C_n^k (n-1)^k}{r_2^k},
\]
which is a contradiction. This shows $\sup_{M_{t}}\rho_{t}\leq r_{2}$. Similarly, we get $\inf_{M_{t}}\rho_{t}\geq r_{1}$. Using Lemma \ref{cj-lem2} and Theorem \ref{Curvature estimate}. we obtain the $C^{1}$ and $C^{2}$ estimates. Higher order estimates follows from the Evans-Krylov theory. Applying the similar argument of \cite{CNSIV}, we get the existence and uniqueness of solution to the equation \eqref{cj-1}.
\end{proof}

\section{Proof of Theorem \ref{Interior estimate}}
Theorem \ref{Interior estimate} can be proved by the similar argument of Theorem \ref{Curvature estimate}. For the reader's convenience, we give a sketch here.

\begin{proof}[Proof of Theorem \ref{Interior estimate}]
By the maximum principle, we assume without loss of generality that $\vp<0$ in $\Omega$. For $x\in\Omega$ and unit vector $\xi$, we consider the function
\[
Q(x,\xi) = \log\vp_{\xi\xi}+\frac{a}{2}|\nabla\vp|^{2}+\frac{A}{2}|x|^{2}+\beta\log(-\vp),
\]
where $a$, $A$ and $\beta$ are constants to be determined later. Suppose that $Q$ achieves its maximum at $(x_{0},\xi_{0})$. Near $x_{0}$, we choose coordinate system such that
\[
\xi_{0} = (1,0,\cdots,0), \ \
\vp_{ij} = \delta_{ij}\vp_{ii}, \ \
\vp_{11} \geq \vp_{22} \geq \cdots \geq \vp_{nn} \ \
\text{at $x_{0}$}.
\]
Thus the new function defined by
\[
\hat{Q}(x) = \log\vp_{11}+\frac{a}{2}|\nabla\vp|^{2}+\frac{A}{2}|x|^{2}+\beta\log(-\vp)
\]
has a maximum at $x_{0}$. Thus,
\begin{equation}\label{Interior estimate eqn 1}
0 = \frac{\vp_{11i}}{\vp_{11}}+a\vp_{i}\vp_{ii}+Ax_{i}+\frac{\beta\vp_{i}}{\vp}.
\end{equation}
Using the similar notations in Section 3, at $x_{0}$, we compute (cf. Lemma \ref{Calculation} and \ref{Bad term})
\begin{equation}\label{Interior estimate eqn 2}
\begin{split}
0 \geq {} & -\frac{2}{\vp_{11}}\sum_{i\geq2}G^{1i,i1}\vp_{11i}^{2}-\frac{F^{ii}\vp_{11i}^{2}}{\vp_{11}^{2}}+\frac{aF^{ii}\vp_{ii}^{2}}{2} \\
& +\frac{A}{2}\sum_{i}F^{ii}-CA-\frac{\beta F^{ii}\vp_{i}^{2}}{\vp^{2}}+\frac{C\beta}{\vp}.
\end{split}
\end{equation}
Combining (\ref{Interior estimate eqn 1}) and (\ref{Interior estimate eqn 2}), and choosing the constant $a$ sufficiently small, we obtain (cf. Lemma \ref{Partial curvature estimate})
\[
|\vp_{ii}| \leq -\frac{C_{a,A,\beta}}{\vp} \ \ \text{for $i\geq2$},
\]
where $C_{a,A}$ is a uniform constant depending on $a$, $A$ and $\beta$. Without loss of generality, we assume that
\begin{equation}\label{Interior estimate eqn 3}
|\vp_{ii}| \leq \delta\vp_{11} \ \ \text{for $i\geq2$},
\end{equation}
where $\delta$ is a constant to be determined later.

On the other hand, by (\ref{Interior estimate eqn 1}) and the Cauchy-Schwarz inequality, we have
\[
-\frac{F^{11}\vp_{111}^{2}}{\vp_{11}^{2}} \geq -Ca^{2}F^{11}\vp_{11}^{2}-CA^{2}F^{11}-\frac{C\beta^{2}F^{11}}{\vp^{2}}
\]
and
\[
-\sum_{i\geq2}\frac{\beta F^{ii}\vp_{i}^{2}}{\vp^{2}}
\geq -\frac{3}{\beta}\sum_{i\geq2}\frac{F^{ii}\vp_{11i}^{2}}{\vp_{11}^{2}}
-\frac{Ca^{2}}{\beta}\sum_{i\geq2}F^{ii}\vp_{ii}^{2}
-\frac{CA^{2}}{\beta}\sum_{i\geq2}F^{ii}.
\]
Substituting these into (\ref{Interior estimate eqn 2}),
\[
\begin{split}
0 \geq {} & -\frac{2}{\vp_{11}}\sum_{i\geq2}G^{1i,i1}\vp_{11i}^{2}
-\left(1+\frac{3}{\beta}\right)\sum_{i\geq2}\frac{F^{ii}\vp_{11i}^{2}}{\vp_{11}^{2}} \\
& +\left(\frac{a}{2}-\frac{Ca^{2}}{\beta}\right)F^{ii}\vp_{ii}^{2}+\left(\frac{A}{2}-\frac{CA^{2}}{\beta}\right)\sum_{i}F^{ii} \\
& -Ca^{2}F^{11}\vp_{11}^{2}-CA^{2}F^{11}-\frac{C\beta^{2}F^{11}}{\vp^{2}} \\
& -CA-\frac{\beta F^{11}\vp_{1}^{2}}{\vp^{2}}+\frac{C\beta}{\vp}.
\end{split}
\]
Choosing $\beta$ sufficiently large and decreasing $a$ if needed, we see that
\[
\begin{split}
0 \geq {} & -\frac{2}{\vp_{11}}\sum_{i\geq2}G^{1i,i1}\vp_{11i}^{2}
-\left(1+\frac{3}{\beta}\right)\sum_{i\geq2}\frac{F^{ii}\vp_{11i}^{2}}{\vp_{11}^{2}} \\
& +\frac{aF^{ii}\vp_{ii}^{2}}{4}-\left(CA^{2}+\frac{C\beta^{2}}{\vp^{2}}\right)F^{11}+\frac{C\beta}{\vp}-CA.
\end{split}
\]
Using (\ref{Interior estimate eqn 3}) and choosing $\delta$ sufficiently small, we obtain (cf. the first inequality of \eqref{Curvature estimate eqn 3})
\[
\left(1+\frac{3}{\beta}\right)\sum_{i\geq2}\frac{F^{ii}\vp_{11i}^{2}}{\vp_{11}^{2}}
\leq -\frac{2}{\vp_{11}}\sum_{i\geq2}G^{1i,i1}\vp_{11i}^{2}+\frac{a}{6}F^{ii}\vp_{ii}^{2}+\frac{C\beta^{2}F^{11}}{\vp^{2}}.
\]
It then follows that
\begin{equation}\label{Interior estimate eqn 4}
0 \geq \frac{F^{11}\vp_{11}^{2}}{C}-\frac{CF^{11}}{\vp^{2}}+\frac{C}{\vp}-C.
\end{equation}
By  (\ref{Interior estimate eqn 3}) and \cite[Lemma 2.2]{HMW10},
\[
F^{11}\vp_{11}^{2} \geq G^{nn}\vp_{11}^{2} \geq \frac{G^{nn}\eta_{nn}\vp_{11}}{C} \geq \frac{\vp_{11}}{C}.
\]
Combining this with (\ref{Interior estimate eqn 4}), we obtain $(-\vp)^{\beta}\vp_{11}\leq C$, as required.
\end{proof}


\begin{thebibliography}{99}

\bibitem{BK74}  I. Ja. Bakelman and B. E. Kantor, {\em Existence of a hypersurface homeomorphic to the sphere in Euclidean space with a given mean curvature,  (Russian)} Geometry and topology, No. 1 (Russian), pp. 3--10. Leningrad. Gos. Ped. Inst. im. Gercena, Leningrad, 1974.

\bibitem{CNSI} L. A. Caffarelli, L. Nirenberg and J. Spruck, {\em Dirichlet problem for nonlinear second order elliptic equations. I. Monge-Amp\`{e}re equation}, Commun. Pure Appl. Math. {\bf 37} (1984), 369--402.

\bibitem{CNSIII} L. A. Caffarelli, L. Nirenberg and J. Spruck, {\em The Dirichlet problem for nonlinear second-order elliptic equations. III. Functions of the eigenvalues of the Hessian}, Acta Math. {\bf 155} (1985), no. 3-4, 261--301.

\bibitem{CNSIV} L. A. Caffarelli, L. Nirenberg and J. Spruck, {\em Nonlinear second order elliptic equations. IV. Starshaped
            compact Weingarten hypersurfaces}, Current topics in partial differential equations, 1--26. Kinokuniya, Tokyo, 1986.

\bibitem{CW01} K.-S. Chou and X-.J. Wang {\em A variational theory of the Hessian equation}, Comm. Pure Appl. Math. {\bf 54} (2001), no. 9, 1029--1064.

\bibitem{FWW10} J. Fu, Z. Wang and D. Wu, {\em Form-type Calabi-Yau equations}, Math. Res. Lett. {\bf 17} (2010), no. 5, 887--903.

\bibitem{FWW15} J. Fu, Z. Wang and D. Wu, {\em Form-type equations on K\"ahler manifolds of nonnegative orthogonal bisectional curvature}, Calc. Var. Partial Differential Equations {\bf 52} (2015), no. 1-2, 327--344.

\bibitem{Gauduchon84} P. Gauduchon, {\em La 1-forme de torsion d'une vari\'et\'e hermitienne compacte}, Math. Ann. {\bf 267} (1984), no. 4, 495--518.

\bibitem{Guan99} B. Guan, {\em The Dirichlet problem for Hessian equations on Riemannian manifolds}, Calc. Var. Partial Differential Equations {\bf 8} (1999), 45--69.

\bibitem{GG02} B. Guan and P. Guan, {\em Convex hypersurfaces of prescribed curvatures}, Ann. of Math. (2) {\bf 156} (2002), no. 2, 655--673.

\bibitem{GJ15} B. Guan and H. Jiao, {\em Second order estimates for Hessian type fully nonlinear elliptic equations on Riemannian manifolds}, Calc. Var. Partial Differential Equations {\bf 54} (2015), no. 3, 2693--2712.

\bibitem{GLL12} P. Guan, J. Li and Y. Li, {\em Hypersurfaces of prescribed curvature measure}, Duke Math. J. {\bf 161} (2012), no. 10, 1927--1942.

\bibitem{GLM09} P. Guan, C. Lin and X.-N. Ma, {\em The existence of convex body with prescribed curvature measures}, Int. Math. Res. Not. IMRN {\bf 2009}, no. 11, 1947--1975.

\bibitem{GRW15} P. Guan, C. Ren and Z. Wang, {\em Global $C^2$ estimates for convex solutions of curvature equations},
    Commun. Pure Appl. Math. {\bf 68} (2015), 1927--1942.

\bibitem{HL13} F. R. Harvey and H. B. Lawson, Jr., {\em $p$-convexity, $p$-plurisubharmonicity and the Levi problem},
    Indiana Univ. Math. J. {\bf 62} (2013), no. 1, 149--169.

\bibitem{HMW10} Z. Hou, X.-N. Ma and D. Wu, {\em A second order estimate for complex Hessian equations on a compact K\"ahler manifold},
    Math. Res. Lett. {\bf 17} (2010), no. 3, 547--561.

\bibitem{Ivochkina90} N. M. Ivochkina, {\em Solution of the Dirichlet problem for equations of $m$th order curvature. (Russian)}, Mat. Sb. {\bf 180} (1989), no. 7, 867--887, 991; translation in Math. USSR-Sb. {\bf 67} (1990), no. 2, 317--339.

\bibitem{Ivochkina91} N. M. Ivochkina, {\em The Dirichlet problem for the curvature equation of order $m$}, Algebra i Analiz {\bf 2} (1990), no. 3, 192--217; translation in Leningrad Math. J. {\bf 2} (1991), no. 3, 631--654.

\bibitem{LRW16} M. Li, C. Ren and Z. Wang, {\em An interior estimate for convex solutions and a rigidity theorem}, J. Funct. Anal. {\bf 270} (2016), no. 7, 2691--2714.

\bibitem{Pogorelov78} A. V. Pogorelov, {\em The Minkowski Multidimensional Problem}, John Wiley, 1978.

\bibitem{Popovici15} D. Popovici, {\em Aeppli cohomology classes associated with Gauduchon metrics on compact complex manifolds}, Bull. Soc. Math. France {\bf 143} (2015), no. 4, 763--800.

\bibitem{RW19} C. Ren and Z. Wang, {\em On the curvature estimates for Hessian equations}, Amer. J. Math. {\bf 141} (2019), no. 5, 1281--1315.

\bibitem{RW20} C. Ren and Z. Wang, {\em The global curvature estimate for the $n-2$ Hessian equation}, preprint, arXiv: 2002.08702.

\bibitem{Sha86} J. P. Sha, {\em p-convex Riemannian manifolds}, Invent. Math. {\bf 83} (1986), no. 3, 437--447.

\bibitem{Sha87} J. P. Sha, {\em Handlebodies and $p$-convexity}, J. Differential Geom. {\bf 25} (1987), no. 3, 353--361.

\bibitem{SX17} J. Spruck and L. Xiao, {\em A note on starshaped compact hypersurfaces with a prescribed scalar curvature in space forms},
    Rev. Mat. Iberoam. {\bf 33} (2017), 547--554.

\bibitem{Szekelyhidi18} G. Sz\'ekelyhidi, {\em Fully non-linear elliptic equations on compact Hermitian manifolds}, J. Differential Geom. {\bf 109} (2018), no. 2, 337--378.

\bibitem{STW17} G. Sz\'ekelyhidi, V. Tosatti and B. Weinkove, {\em Gauduchon metrics with prescribed volume form}, Acta Math. {\bf 219} (2017), no. 1, 181--211.

\bibitem{TW17} V. Tosatti and B. Weinkove, {\em The Monge-Amp\`ere equation for $(n-1)$-plurisubharmonic functions on a compact K\"ahler manifold}, J. Amer. Math. Soc. {\bf 30} (2017), no. 2, 311--346.

\bibitem{TW19} V. Tosatti and B. Weinkove, {\em Hermitian metrics, $(n-1,n-1)$ forms and Monge-Amp\`ere equations}, J. Reine Angew. Math. {\bf 755} (2019), 67--101.

\bibitem{TrWe83} A. E. Treibergs and S. W. Wei, {\em Embedded hyperspheres with prescribed mean curvature}, J. Differential Geom. {\bf 18} (1983), no. 3, 513--521.

\bibitem{Wu87} H. Wu, {\em Manifolds of partially positive curvature}, Indiana Univ. Math. J. {\bf 36} (1987), no. 3, 525--548.

\end{thebibliography}
\end{document}